\documentclass[12pt,leqno,draft]{article}
\usepackage{amsfonts}
\usepackage{amsmath, amsthm, amsfonts, amssymb, color, mathrsfs, dsfont}
\usepackage{mathrsfs}
\usepackage{color}
\newtheorem{thm}{Theorem}[section]

\newtheorem{lem}[thm]{Lemma}

\newtheorem{rem}[thm]{Remark}
\theoremstyle{definition}
\newtheorem{defn}{Definition}[section]
\newcommand{\scr}[1]{\mathscr #1}
\definecolor{wco}{rgb}{0.5,0.2,0.3}

\numberwithin{equation}{section} \theoremstyle{remark}

\newcommand{\ua}{\uparrow}

\title{{\bf Distribution Dependent SDEs with H\"{o}lder Continuous Drift and $\alpha$-Stable Noise}\footnote{Supported in
 part by  NNSFC (11801406).} }
\author{
{\bf     Xing Huang $^{1,a)}$, Fen-Fen Yang $^{1,b)}$   }\\
\footnotesize{  1)Center for Applied Mathematics, Tianjin University, Tianjin 300072, China}\\
\footnotesize{ $^{a)}$xinghuang@tju.edu.cn, $^{b)}$yangfenfen@tju.edu.cn}
}

\begin{document}
\allowdisplaybreaks
\def\R{\mathbb R}  \def\ff{\frac} \def\ss{\sqrt} \def\B{\mathbf
B}
\def\N{\mathbb N} \def\kk{\kappa} \def\m{{\bf m}}
\def\ee{\varepsilon}\def\ddd{D^*}
\def\dd{\delta} \def\DD{\Delta} \def\vv{\varepsilon} \def\rr{\rho}
\def\<{\langle} \def\>{\rangle} \def\GG{\Gamma} \def\gg{\gamma}
  \def\nn{\nabla} \def\pp{\partial} \def\E{\mathbb E}
\def\d{\text{\rm{d}}} \def\bb{\beta} \def\aa{\alpha} \def\D{\scr D}
  \def\si{\sigma} \def\ess{\text{\rm{ess}}}
\def\beg{\begin} \def\beq{\begin{equation}}  \def\F{\scr F}
\def\Ric{\text{\rm{Ric}}} \def\Hess{\text{\rm{Hess}}}
\def\e{\text{\rm{e}}} \def\ua{\underline a} \def\OO{\Omega}  \def\oo{\omega}
 \def\tt{\tilde} \def\Ric{\text{\rm{Ric}}}
\def\cut{\text{\rm{cut}}} \def\P{\mathbb P} \def\ifn{I_n(f^{\bigotimes n})}
\def\C{\scr C}   \def\G{\scr G}   \def\aaa{\mathbf{r}}     \def\r{r}
\def\gap{\text{\rm{gap}}} \def\prr{\pi_{{\bf m},\varrho}}  \def\r{\mathbf r}
\def\Z{\mathbb Z} \def\vrr{\varrho} \def\ll{\lambda}
\def\L{\scr L}\def\Tt{\tt} \def\TT{\tt}\def\II{\mathbb I}
\def\i{{\rm in}}\def\Sect{{\rm Sect}}  \def\H{\mathbb H}
\def\M{\scr M}\def\Q{\mathbb Q} \def\texto{\text{o}} \def\LL{\Lambda}
\def\Rank{{\rm Rank}} \def\B{\scr B} \def\i{{\rm i}} \def\HR{\hat{\R}^d}
\def\to{\rightarrow}\def\l{\ell}\def\iint{\int}
\def\EE{\scr E}\def\no{\nonumber}
\def\A{\scr A}\def\V{\mathbb V}\def\osc{{\rm osc}}
\def\BB{\scr B}\def\Ent{{\rm Ent}}\def\3{\triangle}
\def\U{\scr U}\def\8{\infty}\def\1{\lesssim}
\def\R{\mathbb R}  \def\ff{\frac} \def\ss{\sqrt} \def\B{\mathbf
B} \def\W{\mathbb W}
\def\N{\mathbb N} \def\kk{\kappa} \def\m{{\bf m}}
\def\ee{\varepsilon}\def\ddd{D^*}
\def\dd{\delta} \def\DD{\Delta} \def\vv{\varepsilon} \def\rr{\rho}
\def\<{\langle} \def\>{\rangle} \def\GG{\Gamma} \def\gg{\gamma}
  \def\nn{\nabla} \def\pp{\partial} \def\E{\mathbb E}
\def\d{\text{\rm{d}}} \def\bb{\beta} \def\aa{\alpha} \def\D{\scr D}
  \def\si{\sigma} \def\ess{\text{\rm{ess}}}
\def\beg{\begin} \def\beq{\begin{equation}}  \def\F{\scr F}
\def\Ric{\text{\rm{Ric}}} \def\Hess{\text{\rm{Hess}}}
\def\e{\text{\rm{e}}} \def\ua{\underline a} \def\OO{\Omega}  \def\oo{\omega}
 \def\tt{\tilde} \def\Ric{\text{\rm{Ric}}}
\def\cut{\text{\rm{cut}}} \def\P{\mathbb P} \def\ifn{I_n(f^{\bigotimes n})}
\def\C{\scr C}      \def\aaa{\mathbf{r}}     \def\r{r}
\def\gap{\text{\rm{gap}}} \def\prr{\pi_{{\bf m},\varrho}}  \def\r{\mathbf r}
\def\Z{\mathbb Z} \def\vrr{\varrho} \def\ll{\lambda}
\def\L{\scr L}\def\Tt{\tt} \def\TT{\tt}\def\II{\mathbb I}
\def\i{{\rm in}}\def\Sect{{\rm Sect}}  \def\H{\mathbb H}
\def\M{\scr M}\def\Q{\mathbb Q} \def\texto{\text{o}} \def\LL{\Lambda}
\def\Rank{{\rm Rank}} \def\B{\scr B} \def\i{{\rm i}} \def\HR{\hat{\R}^d}
\def\to{\rightarrow}\def\l{\ell}\def\iint{\int}
\def\EE{\scr E}\def\Cut{{\rm Cut}}
\def\A{\scr A} \def\Lip{{\rm Lip}}
\def\BB{\scr B}\def\Ent{{\rm Ent}}\def\L{\scr L}
\def\R{\mathbb R}  \def\ff{\frac} \def\ss{\sqrt} \def\B{\mathbf
B}
\def\N{\mathbb N} \def\kk{\kappa} \def\m{{\bf m}}
\def\dd{\delta} \def\DD{\Delta} \def\vv{\varepsilon} \def\rr{\rho}
\def\<{\langle} \def\>{\rangle} \def\GG{\Gamma} \def\gg{\gamma}
  \def\nn{\nabla} \def\pp{\partial} \def\E{\mathbb E}
\def\d{\text{\rm{d}}} \def\bb{\beta} \def\aa{\alpha} \def\D{\scr D}
  \def\si{\sigma} \def\ess{\text{\rm{ess}}}
\def\beg{\begin} \def\beq{\begin{equation}}  \def\F{\scr F}
\def\Ric{\text{\rm{Ric}}} \def\Hess{\text{\rm{Hess}}}
\def\e{\text{\rm{e}}} \def\ua{\underline a} \def\OO{\Omega}  \def\oo{\omega}
 \def\tt{\tilde} \def\Ric{\text{\rm{Ric}}}
\def\cut{\text{\rm{cut}}} \def\P{\mathbb P} \def\ifn{I_n(f^{\bigotimes n})}
\def\C{\scr C}      \def\aaa{\mathbf{r}}     \def\r{r}
\def\gap{\text{\rm{gap}}} \def\prr{\pi_{{\bf m},\varrho}}  \def\r{\mathbf r}
\def\Z{\mathbb Z} \def\vrr{\varrho} \def\ll{\lambda}
\def\L{\scr L}\def\Tt{\tt} \def\TT{\tt}\def\II{\mathbb I}
\def\i{{\rm in}}\def\Sect{{\rm Sect}}  \def\H{\mathbb H}
\def\M{\scr M}\def\Q{\mathbb Q} \def\texto{\text{o}} \def\LL{\Lambda}
\def\Rank{{\rm Rank}} \def\B{\scr B} \def\i{{\rm i}} \def\HR{\hat{\R}^d}
\def\to{\rightarrow}\def\l{\ell}
\def\8{\infty}\def\I{1}\def\U{\scr U}
\maketitle

\begin{abstract} In this paper, the existence and uniqueness of the distribution dependent SDEs with H\"{o}lder continuous drift driven by $\alpha$-stable process is investigated. Moreover, by using Zvonkin type transformation, the convergence rate of Euler-Maruyama method and propagation of chaos are also obtained. The results cover the ones in the case of distribution independent SDEs.
\end{abstract} \noindent
 AMS subject Classification: 60H35, 41A25, 60H10, 60C30.   \\
\noindent
 Keywords:  Distribution dependent SDEs, H\"{o}lder continuous, Zvonkin type transformation, Euler-Maruyama method, $\alpha$-stable process.

\section{Introduction}
Distribution dependent stochastic differential equations (SDEs for abbreviation), also called McKean-Vlasov SDE, can be used to characterize  nonlinear Fokker-Planck equations (see \cite{Mc,BLPR}). Recently, there are many results on Distribution dependent SDEs
(see \cite{MV,SZ} and references within).
Existence and uniqueness of McKean-Vlasov SDEs with regular
coefficients have been investigated extensively (see e.g.
\cite{BMP,CD,MV,SZ, Wangb}). Meanwhile, the strong wellposedness of
McKean-Vlasov SDEs with irregular coefficients has also received
much attention (see, for example, \cite{Ch,HW,RZ}, where, in \cite{Ch},
the dependence of laws  is of integral type and the diffusion is
non-degenerate, and \cite{HW} is concerned with the integrability
condition but excluding   linear growth of the drift). For weak
wellposedness of McKean-Vlasov SDEs, we refer to e.g.
\cite{HW,LM,MS,MV}. \cite{RW,W19} studied the Lion's derivative and ergodicity for  SDEs driven by Brownian motion. \cite{So} investigated the derivative formula and gradient estimate for McKean-Vlasov SDEs driven by jump process (See \cite{ABRS,BLM,CM,DEG,EGZ} for more results on McKean-Vlasov SDEs).

Recently, the convergence rate of Euler-Maruyama (EM for short) method for SDEs with irregular coefficients has attracted much attention. For instance, \cite{GR,Y} revealed the convergence rate in $L^1$ and $L^p$-norm sense for a range of SDEs, where the drift term is Lipschitzian and the diffusion term is H\"{o}lder continuous with respect to spatial variable. In addition, by using the Yamada-Watanabe approximation and heat kernel estimate, \cite{NT2} studied the strong convergence rate for a class of non-degenerate SDEs with bounded drift term satisfies weak monotonicity and is of bounded variation with respect to Gaussian measure and the diffusion term is H\"{o}lder continuous.

Quite recently, by Zvonkin transformation \cite{ZV}, the convergence rate of EM method for the SDEs with singular drift are investigated extensively. For instance, \cite{BHY} discussed the case with Dini continuous drifts;
\cite{PT} obtained the strong convergence rate of EM method with bounded H\"{o}lder continuous drift driven by truncated symmetric $\alpha$-stable process, see also \cite{MX}  and \cite{HL} for the  symmetric $\alpha$-stable process.
 As to the distribution dependent SDEs, \cite{Z2} proved the convergence of the EM scheme under linear growth condition by a discretized version of Krylov's estimate. \cite{BH} extended the results of \cite{PT} and \cite{GR} to the distribution dependent SDEs driven by Brownian motion.

In this paper, we investigate the existence and uniqueness of distribution dependent SDEs with bounded and H\"{o}lder continuous drifts, where the noise is $\alpha$-stable process. Due to the distribution dependence, we adopt an approximation technique constructing by a sequence of classical SDEs and Skorohod representation theorem to prove the existence of the weak solution. As to the pathwise uniqueness, we still use the Zvonkin transform which depend on the distribution of one solution to make two solutions be regular ones.

Since the SDE is distribution-dependent, we exploit the
stochastic interacting particle systems to approximate it. We will apply a common Zvonkin's transform depending on the distribution of the real solution to make the numerical SDE and interacting particle systems be regular ones, from which the strong convergence rate is obtained.

The paper is organized as follows. In Section 2, we recall some preliminaries on symmetric $\alpha$-stable process and the Poisson random measure. In Section 3, the existence and uniqueness  for  the distribution dependent SDEs with H\"{o}lder continuous drift driven by $\alpha$-stable process are established. Finally,  by using Zvonkin type transformation, the convergence rate of EM Scheme for SDEs are investigated in Section 4.
\section{Some Preparations}
\subsection{Symmetric $\alpha$-stable process}
Before moving on, we firstly recall some knowledge on symmetric $\alpha$-stable process and the Poisson random measure  (see \cite{HL,K,L,P} for more details). Recall that a $\mathbb{R}^{d}$-valued L\'{e}vy process $L_t$ is called $d$-dimensional symmetric $\alpha$-stable process if the L\'{e}vy symbol $\Psi$ has the following representation:
\beg{equation*}
\Psi(u)=\int_{\mathbb{R}^{d}}[1-\cos\langle u, x\rangle]\nu(\d x),
\end{equation*}
where
\beg{equation*}
\nu(D)=\int_{S}\mu(\d \xi)\int_{0}^{\infty} \mathds{1}_{D}(r\xi)\frac{\d r}{r^{1+\alpha}}, \quad D\in\B(\mathbb{R}^{d}),
\end{equation*}
$\alpha\in(0,2)$, $S=\{x\in\mathbb{R}^{d}, |x|=1\}$ and $\mu$ is a finite symmetric measure on $(S,\B(S))$, i.e. $\mu(A)=\mu(-A)$, for any $A\in \B(S)$.

The Poisson random measure $N$ associated to $L$ is defined as follows:
\beg{equation*}
N([0,t], U)=\sum_{0\leq s\leq t} \mathds{1}_{U}(\Delta L_s), \quad U\in\B\left(\mathbb{R}^{d}\backslash\{0\}\right), t\geq 0.
\end{equation*}
Here $\Delta L(s)=L_s-L_{s-}$ denotes the jump size of $L$ at time $s\geq0$. The compensated Poisson random measure $\tilde{N}$ is defined by
\beg{equation*}
\tilde{N}([0,t], U)=N([0,t], U)-t\nu(U), \quad U\in\B\left(\mathbb{R}^{d}\backslash\{0\}\right), 0\notin\bar{U}, t\geq 0.
\end{equation*}
It follows from the L\'{e}vy-It\^{o} decomposition that
\beg{equation*}
L_t=\int_{0}^{t}\int_{|x|\leq 1}x\tilde{N}(\d s, \d x)+\int_{0}^{t}\int_{|x|> 1}x N(\d s, \d x), \quad t\geq 0.
\end{equation*}



For convenience, we introduce some notations. Let $\|\cdot\|$ denote the operator norm for a bounded linear operator. For $k\in\mathbb{N}$ and $\beta\in(0,1)$,
denote by $C_{b}^{k+\beta}\left(\mathbb{R}^{d}\right)$ the set of $\mathbb{R}^{d}$-valued bounded functions, which have up to $k$-ordered continuous derivative and the $k$-th derivative is $\beta$ H\"{o}lder continuous. The norm is
\beg{equation*}
\|f\|_{k+\beta}:=\sum_{i=0}^{k}\sup_{x\in\mathbb{R}^{d}}\|\nabla^{i} f(x)\|+\sup_{x\neq y}\frac{\|\nabla ^{k} f(x)-\nabla ^{k} f(y)\|}{|x-y|^{\beta}},\quad  f\in C_{b}^{k+\beta}\left(\mathbb{R}^{d}\right).
\end{equation*}
In particular, $C_{b}^{0} \left(\mathbb{R}^{d}\right) $ means the set of $\mathbb{R}^{d}$-valued bounded functions, equipped the norm $\|f\|_{\infty}:=\sup_{x\in\mathbb{R}^{d}}|f(x)|$, and we usually denote $C_{b}$. Let $T>0$, for a function $f$ defined on $[0,T]\times \mathbb{R}^d$, let $\|f\|_{T,\infty}=\sup_{t\in[0,T],x\in\mathbb{R}^d}|f(t,x)|$.
\subsection{Distribution dependent SDEs}
Let $\mathscr{P}$ be the
collection of all probability measures on $\R^d$. For   $p\geq 1$, if
$\mu(|\cdot|^p):=\int_{\R^d}|x|^p\mu(\d x)<\8$, we formulate
$\mu\in\mathscr{P}_p$.  For $\mu,\bar{\mu}\in\mathscr{P}_p$,
$p\geq 1,$ the $\mathbb{W}_p$-Wasserstein distance between $\mu$ and
$\nu$ is defined by
\begin{equation*}
\mathbb{W}_p(\mu,\bar{\mu})=\inf_{\pi\in\mathcal
{C}(\mu,\bar{\mu})}\Big(\int_{\R^d\times\R^d}|x-y|^p\pi(\d x,\d y)\Big)^{\ff{1}{
p}},
\end{equation*}
where $\mathcal {C}(\mu,\bar{\mu})$ stands for the set of all couplings of
$\mu$ and $\bar{\mu}$. As for a random variable $\xi,$ its law is
written by $\mathscr{L}_\xi$, and write $\mathscr{L}_\xi|_{\P}$ as the distribution of $\xi$ under $\P$.

Consider the following McKean-Vlasov SDE on $\R^d$
\beq\label{E1}
\d
X_t= b(X_t, \L_{X_t})\d t +\d L_t,~~~t\ge0.
\end{equation}
{\begin{defn}\label{def1} A c\`{a}l\`{a}g adapted process $(X_t)_{t\geq
0}$ on $\mathbb{R}^d$ is called a (strong) solution of \eqref{E1}, if
$\P$-a.s.
\begin{align}\label{sol}
X_t=X_0+\int_0^tb(X_s,\L_{X_s})\,\d s+L_t, \quad t\ge0.
\end{align}
We call the strong uniqueness in $\scr P_{\theta}$ for some $\theta\in[1,\alpha)$, if for any $\F_0$-measurable random variable $X_0$ with $\L_{X_0}\in\scr P_{\theta}$, there exists a unique $X_t$ satisfy \eqref{sol} and $\E|X_t|^\theta<\infty$.

(2) A couple $(\tilde{X}_t
, \tilde{L}_t)_{t\geq 0}$ is called a weak solution to \eqref{sol}, if $\tilde{L}$
is a
$d$-dimensional symmetric $\alpha$-stable process with respect to a complete filtration probability space
$(\tilde{\Omega}, \{\tilde{\F}_t\}_{t\geq 0}, \tilde{\P})$, and \eqref{sol} holds for $(\tilde{X}_t
, \tilde{L}_t)_{t\geq 0}$ in place of $(X_t
, L_t)_{t\geq 0}$.

(3) \eqref{sol} is said to have weak uniqueness in $\scr P_\theta$ for some $\theta\in[1,\alpha)$, if any two weak solutions of
the equation from  common initial distribution in $\scr P_\theta$ are equal in law.
\end{defn}
Throughout this paper, we assume that
\beg{enumerate}
\item[{\bf (H1)}] For fixed $\alpha\in(1,2)$, there exists a positive constant $C_{\alpha}>0$ such that
\beg{equation*}
\Psi(u)\geq C_{\alpha}|u|^{\alpha}, \quad u\in\mathbb{R}^{d}.
\end{equation*}
\end{enumerate}
\beg{enumerate}
\item[{\bf (H2)}] $\|b\|_\infty:=\sup_{x\in\R^d,\mu\in\scr P}|b(x,\mu)|<\infty$, and there exists constants $\beta\in(0,1)$ satisfying $2\beta+\alpha>2$, $K>0$ and $\kappa\in[1,\alpha)$ such that
\begin{align}\label{Holder}
|b(x,\mu)-b(y,\bar{\mu})|\leq K(\mathbb{W}_\kappa(\mu,\bar{\mu})+|x-y|^\beta),\ \ \mu,\bar{\mu}\in\scr P_\kappa, x,y\in\mathbb{R}^d.
\end{align}
\end{enumerate}
See \cite[Remark 1.1]{HL} for examples such that {\bf(H1)} holds.

\section{Existence and uniqueness}
\subsection{Weak Solution}
We will use the tightness and Skorohod representation theorem to prove the weak existence. The idea of the proof of the following theorem comes from \cite[Proof of Theorem 4.1]{Z1}  (see also \cite[Proof of Theorem 4.7]{J} and \cite{HW,MV} for the case with Gaussian noise).
\begin{thm}\label{ws} Assume that $b$ is bounded measurable. Then for any $\mu_0\in\scr P_\kappa$, 
\eqref{E1} has a weak solution with initial distribution $\mu_0$.
\end{thm}
\begin{proof}
Let $0\le \rr\in C_0^\infty(\R^d)$ with support contained in $\{x: |x|\le 1\}$ such that $\int_{\R^d} \rr(x)\d x=1.$
For any $n\ge 1$, let $\rr_n(x)= n^{d} \rr(nx)$ and define
\begin{equation}\begin{split}\label{approx}
&b^n(x,\mu)=\int_{\R^d} b(x',\mu) \rr_n (x-x')\d x',\ \ (x,\mu)\in \R^d\times\scr P.
\end{split}\end{equation}
The by {\bf(H2)}, for any $n\geq 1$, there exists a constant $C_n>0$ such that
$$|b^n(x,\mu)-b^n(y,\bar{\mu})|\leq C_n(|x-y|+\W_\kappa(\mu,\bar{\mu})), \ \ (x,\mu), (y,\bar{\mu})\in \R^d\times\scr P_\kappa.$$
Moreover, it holds
\begin{align}\label{bnc}
|b^n(x,\mu)-b(x,\bar{\mu})|&\leq|b^n(x,\mu)-b^n(x,\bar{\mu})|+|b^n(x,\bar{\mu})-b(x,\bar{\mu})|\\ \nonumber
&\leq K\W_\kappa(\mu,\bar{\mu}))+|b^n(x,\bar{\mu})-b(x,\bar{\mu})|.
\end{align}
For any $n\geq 1$, define
\begin{align}\label{Yn}
\d X^{n}_t=b^n(X^{n}_t,\L_{X^{n}_t})\d t+\d L_t,
\end{align}
with $\L_{X^{n}_0}=\mu_0$.
Then use a distribution iteration method as in the case with Gaussian noise (\cite{Wangb}), it is not difficult to see that \eqref{Yn} has a solution $\{X^{n}\}_{n\geq 1}$ on $[0,T]$ with $\L_{X^n_t}\in \scr P_\kappa$.

Let $\tilde{b}^{n}_t(x)=b(x,\L_{X^{n}_t})$. Then \eqref{Yn} can be rewritten as
\begin{align}\label{Yn0}
\d X^{n}_t=\tilde{b}^n_t(X^{n}_t)\d t+\d L_t.
\end{align}

Let $\mathbb{D}$ be the space of all $\mathbb{R}^d$-valued c\`{a}l\`{a}g functions on $[0,T]$ equipped with the Skorohod topology such that $\mathbb{D}$ is a Polish space.
Set $$H_s^n=\int_0^s\tilde{b}^n_t(X^{n}_t)\d t,\ \ s\in[0,T].$$
Since $b$ is bounded, it is clear that
$$\sup_{s\in[0,T]}|H_s^n|\leq T\|b\|_{\infty}$$
for any $n\geq 1$. Moreover, for any $\varepsilon>0$ and bounded stopping time $\tau$
$$|H_{t\wedge\tau}^n-H^n_{t\wedge(\tau+\varepsilon)}|\leq \varepsilon\|b\|_{\infty},\ \ t\in[0,T].$$
Thus, $\{H^n_\cdot\}_{n\geq 1}$ in $\mathbb{D}$ is tight, and so does $\{H^n_\cdot,L_\cdot\}_{n\geq 1}$. So there exists a subsequence still denoted by $\{H^n_\cdot,L_\cdot\}_{n\geq 1}$ such that the distribution of $\{H^n_\cdot,L_\cdot\}_{n\geq 1}$ is weakly convergent in $\mathbb{D}\times \mathbb{D}$, which implies weak convergence of the distribution of $\{X^{n}_\cdot,L_\cdot\}_{n\geq 1}$ in $\mathbb{D}\times \mathbb{D}$. Then, by Skorohod's representation theorem,
there exists a probability space $(\tilde{\Omega},\tilde{\F}, \tilde{\P})$ and $\mathbb{D}\times \mathbb{D}$-valued stochastic processes $\{\tilde{X}^{n}_\cdot, \tilde{L}^n_\cdot)$, $\{\tilde{X}_\cdot, \tilde{L}_\cdot)$ such that $\L_{(X^{n}_\cdot, L_\cdot)}|\P=\L_{(\tilde{X}^{n}_\cdot, \tilde{L}^n_\cdot)}|_{\tilde{\P}}$, and $\tilde{\P}$-a.s. $(\tilde{X}^{(n)}_\cdot, \tilde{L}^n_\cdot)$ converges to $(\tilde{X}_\cdot, \tilde{L}_\cdot)$ as $n\to\infty$, which implies that for any $t\in[0,T]$, $\L_{\tilde{X}^{n}_t}|\tilde{\P}$ weakly converges to $\L_{\tilde{X}_t}|\tilde{\P}$. In particular, $\tilde{L}$ is still a symmetric $\alpha$-stable L\'{e}vy process with respect to the complete filtration $\tilde{\F}_t=\overline{\sigma\{\tilde{X}_s, \tilde{L}_s,\ \ s\leq t\}}^{\tilde{\P}}$ and has the same symbol as $L$, and
\begin{align}\label{Yn0'}
\d \tilde{X}^{n}_t=b(\tilde{X}^{n}_t, \L_{\tilde{X}^{n}_t})\d t+\d \tilde{L}^n_t, \ \ \tilde{X}^{n}_0=\tilde{X}_0
\end{align}
with $\L_{\tilde{X}_0}|\tilde{\P}=\L_{X_0^n}|\P$. Next, we only need to take limit in \eqref{Yn0'}.

For any $n\ge m\ge 1$, we have
$$\int_{0}^s | b^n({\tilde X}^n_t,\L_{{\tilde X}^n_t})-b({\tilde X}_t,\L_{{\tilde X}_t})|\,\d t\le I_1(s)+ I_2(s)+I_3(s),$$ where
\begin{align*}
&I_1(s):= \int_{0}^s| b^n({\tilde X}^n_t,\L_{{\tilde X}^n_t})-b^{m}({\tilde X}^n_t,\L_{{\tilde X}_t})|\,\d t,\\
&I_2(s):=\int_{0}^s|b^{m}({\tilde X}^n_t,\L_{{\tilde X}_t})-b^{m}({\tilde X}_t,\L_{{\tilde X}_t})|\,\d t,\\
&I_3(s):= \int_{0}^s| b^{m}({\tilde X}_t,\L_{{\tilde X}_t})-b({\tilde X}_t,\L_{{\tilde X}_t})|\,\d t.
\end{align*}
Below we estimate these $I_i(s)$ respectively. For simplicity, let $\tilde{\mu}_t=\L_{\tilde{X}_t}$ and $\tilde{\mu}^n_t=\L_{\tilde{X}^n_t}$.

Firstly, since $\|b^n\|_\infty\leq \|b\|_{\infty}$, applying Krylov's estimate in \cite[Theorem 3.1]{Z1} and Chebyshev's inequality, we arrive at for any $p>\frac{d}{\alpha}\vee 1$ and $q>\frac{p\alpha}{p\alpha-d}$,
\begin{align*}
\tilde{\P}(\sup_{s\in[0,T]}I_1(s)\geq\frac{\varepsilon}{3})&\leq \frac{9}{\varepsilon^2}\E\int_{0}^T1_{\{|\tilde{X}^n_t|\leq R\}}| b^n(\tilde{X}^n_t,\tilde{\mu}^n_t)-b^{m}(\tilde{X}^n_t,\tilde{\mu}_t)|^2\,\d t\\
&+\frac{9}{\varepsilon^2}\E\int_{0}^T1_{\{|\tilde{X}^n_t|> R\}}| b^n(\tilde{X}^n_t,\tilde{\mu}^n_t)-b^{m}(\tilde{X}^n_t,\tilde{\mu}_t)|^2\,\d t\\
&\leq \frac{C}{\varepsilon^2}\left(\int_{0}^T\left(\int_{|x|\leq R}|b^n(x,\tilde{\mu}^n_t)-b^{m}(x,\tilde{\mu}_t)|^{2p}\d x\right)^{q/p}\d t\right)^{\frac{1}{q}}\\
&+\frac{C}{\varepsilon^2}\int_{0}^T\tilde{\P}(|\tilde{X}^n_t|> R)\d t.
\end{align*}
Since $\tilde{X}^n_t$ converges to $\tilde{X}_t$ in probability, it is clear
$$\lim_{n\to\infty} \W_\kappa(\tt\mu_t^n,\mu_t) =0,$$
and $$\lim_{n\to\infty}\tilde{\P}(|\tilde{X}^n_t|> R)\leq\tilde{\P}(|\tilde{X}_t|\geq R).$$
Then it follows from the definition of $b^n$ and \eqref{bnc} that
\begin{align*}
&\lim_{n\to\infty}|b^n(x,\tilde{\mu}^n_t)-b(x,\tilde{\mu}_t)|=0, \ \ a.e. \ \ x\in\mathbb{R}^d.
\end{align*}
So, we may apply the dominated convergence theorem to derive
\begin{equation}\begin{split}\label{I1J}
&\limsup_{n\to\infty}\tilde{\P}(\sup_{s\in[0,T]}I_1(s)\geq\frac{\varepsilon}{3})\\
&\leq \frac{C}{\varepsilon^2}\left(\int_{0}^T\left(\int_{|x|\leq R}|b(x,\tilde{\mu}_t)-b^{m}(x,\tilde{\mu}_t)|^{2p}\d x\right)^{q/p}\d t\right)^{\frac{1}{q}}\\
&+\frac{C}{\varepsilon^2}\int_{0}^T\tilde{\P}(|\tilde{X}_t|\geq R)\d t.
\end{split}\end{equation}
Since $b^{m}$ is bounded and continuous, it follows that
\begin{align*}
&\limsup_{n\to\infty}\tilde{\P}\Big(\sup_{s\in[0,T]}I_2(s)\geq\frac{\varepsilon}{3}\Big)\leq\limsup_{n\to\infty}\frac{3}{\varepsilon}
\E\int_{0}^T|b^{m}(\tilde{X}^n_t,\L_{\tilde{X}_t})-b^{m}(\tilde{X}_t,\L_{\tilde{X}_t})|\,\d t=0.
\end{align*}

Finally, since $\tt X^n_t\to \tt X_t$ in probability, the Krylov's estimate in \cite[Theorem 3.1]{Z1} also holds for $\tt X$ replacing $\tt X^n$. Therefore, inequality \eqref{I1J} holds for $I_3$ replacing $I_1$. In conclusion, we arrive at
\begin{align*}
&\limsup_{n\to\infty}\tilde{\P}\Big(\sup_{s\in[0,T]}\int_{0}^s| b^n(\tilde{X}^n_t,\L_{\tilde{X}^n_t})-b(\tilde{X}_t,\L_{\tilde{X}_t})|\,\d t\geq\varepsilon\Big)\\
&\leq \limsup_{n\to\infty}\sum_{i=1}^3\tilde{\P}\Big(\sup_{s\in[0,T]}I_i(s)\geq\frac{\varepsilon}{3}\Big)\\
&\leq \frac{C}{\varepsilon^2}\left(\int_{0}^T\left(\int_{|x|\leq R}|b(x,\tilde{\mu}_t)-b^{m}(x,\tilde{\mu}_t)|^{2p}\d x\right)^{q/p}\d t\right)^{\frac{1}{q}}\\
&+\frac{C}{\varepsilon^2}\int_{0}^T\tilde{\P}(|\tilde{X}_t|\geq R)\d t
\end{align*}
for any $m>0$ and $R>0$. Then letting first $m\to\infty$ and then $R\to\infty$, we obtain from the dominated convergence theorem that
\begin{align*}
&\limsup_{n\to\infty}\tilde{\P}\Big(\sup_{s\in[0,T]}\int_{0}^s| b^n(\tilde{X}^n_t,\L_{\tilde{X}^n_t})-b(\tilde{X}_t,\L_{\tilde{X}_t})|\,\d t\geq\varepsilon\Big)=0.
\end{align*}
Finally, letting $n$ go to infinity in \eqref{Yn0'}, we have
\begin{align}\label{Xs}
\d \tilde{X}_t=b(\tilde{X}_t, \L_{\tilde{X}_t}|_{\tilde{\P}})\d t+\d \tilde{L}_t.
\end{align}
Thus, $(\tilde{X}, \tilde{L})$ is a weak solution to \eqref{E1}.
\end{proof}
\begin{rem}\label{KIn} Since the Krylov's estimate also holds under some integrable condition on the drift in \cite{Z1}, the existence of weak solution can be proved when $b$ satisfies some integrable condition. We only consider the bounded measurable drift in this paper since the convergence rate of Euler-Maruyama method in  Section 4 can not been obtained under integrable condition.
\end{rem}
\begin{thm}\label{uwo} Assume {\bf(H1)}-{\bf(H2)}. Then \eqref{E1} has weak uniqueness in $\scr P_\kappa$.
\end{thm}
\begin{proof} Let   $(X_t)_{t\ge 0}$ solve \eqref{E1} with $\scr L_{X_0}=\mu_0$, and let   $(\tt X_t,\tt L_t)$ on
$(\tt\OO, \{\tt\F_t\}_{t\ge 0}, \tt\P)$ be a weak solution of \eqref{E1}  such that  $\L_{X_0}|_{\P}= \L_{\tt X_0}|_{\tt\P}=\mu_0$, i.e.
$\tt X_t$ solves
\beq\label{E1'0} \d \tt X_t = b(\tt X_t, \L_{\tt X_t}|_{\tt\P})\d t + \d \tt L_{t},\ \ \ \scr L_{\tt X_0}=\mu_0.\end{equation}
 We aim to   prove   $\L_{X}|_{\P}=\L_{\tt X}|_{\tt\P}$.
  Let $\mu_t= \L_{X_t}|_{\P}$ and
$$\bar b_t(x)= b(x, \mu_t),\ \ x\in\R^d.$$
According to \cite{P}, the stochastic differential equation
\beq\label{E10} \d \bar X_t = \bar b_t(\bar X_t)\d t + \d \tt L_{t},\ \ \bar X_0= \tt X_0 \end{equation}
has a unique solution under {\bf(H1)}-{\bf(H2)}.
According to Yamada--Watanabe \cite{IW}, it also satisfies weak uniqueness. Noting that
$$\d X_t= \bar b_t( X_t)\d t +\d  L_{t},\ \ \L_{X_0}|_{\P}= \L_{\tt X_0}|_{\tt\P},$$ the weak uniqueness of \eqref{E10} implies
\beq\label{HW} \L_{\bar X}|_{\tt\P}= \L_X|_{\P}.\end{equation}
So, \eqref{E10} reduces to
$$ \d \bar X_t = b(\bar X_t, \L_{\bar X_t}|_{\tt\P})\d t + \d \tt L_{t},\ \ \bar X_0=\tt X_0.$$
By the strong uniqueness of \eqref{E1} according to Theorem \ref{uso} below,   we obtain  $\bar X=\tt X$. Therefore,  \eqref{HW} implies $\L_{\tt X}|_{\tt \P} = \L_X|_{\P}$ as wanted.
\end{proof}
\subsection{Strong Solution}
The next lemma characterize the relationship between the existence of weak and strong solution
(see \cite{HW, H1}).
\begin{lem}\label{SS} Let $(\bar\Omega, \{\bar\F_t\}_{t\geq 0},\bar\P)$ and $(\bar{X}_t,L_t)$ be a weak solution to \eqref{E1} with $\mu_t:=\L_{\bar X_t}|_{\bar\P}$. If the SDE
\begin{align}\label{class}
\d X_t= b(X_t,\mu_t)\,\d t+\d L_t,\ \ 0\le t\le T
\end{align}
has a unique strong solution $X_t$ up to life time with $\L_{X_0}=\mu_0$, then   \eqref{E1} has a strong solution.
\end{lem}
\begin{proof} Since $\mu_t= \scr L_{\bar X_t}|_{\bar \P}$,   $\bar{X}_t$ is a weak solution to \eqref{class}. By Yamada-Watanabe principle, the strong uniqueness of \eqref{class} implies the weak uniqueness, so that $X_t$ is nonexplosive with    $\L_{X_t}=\mu_t, t\ge 0$. Therefore, $X_t$ is a strong solution to \eqref{E1}.
\end{proof}
\begin{rem}\label{wr} According to \cite{P}, \eqref{class} has a unique strong solution  under {\bf(H1)}-{\bf(H2)}. This together with Lemma \ref{SS} and Theorem \ref{ws} implies that \eqref{E1} has a strong solution.
\end{rem}
\begin{thm}\label{uso} Assume {\bf(H1)}-{\bf(H2)}. Let $X$ and $Y$ be two solutions to \eqref{E1} in $\scr P_\kappa$ with $X_0=Y_0$. Then $\P$-a.s. $X=Y$.
\end{thm}
\begin{proof} Let $\mu_t=\L_{X_t}, \bar{\mu}_t=\L_{Y_t}, t\in [0,T].$ Then $\mu_0=\bar{\mu}_0$.
Let $$b_t^\mu(x)= b(x, \mu_t),\ \ b_t^{\bar{\mu}}(x)= b(x, \bar{\mu}_t),\ \ (t,x)\in [0,T]\times\R^d.$$  Then it holds
\beq\label{E1'}\beg{split}
&\d X_t= b^\mu_t(X_t)\,\d t+ \d L_t,\\
&\d Y_t= b_t^{\bar{\mu}}(Y_t)\d t +\d L_t.\end{split}
\end{equation}
For  $\ll>0 $, consider the following PDE for
$u^{\ll,\mu}:[0,T]\times\R^d\to\R^d$:
\begin{equation}\label{A10}
\partial_tu^{\ll,\mu}_t+\L u^{\ll,\mu}_t+\nn_{b_t^{\mu}}u^{\ll,\mu}_t+
b_t^{\mu}=\ll u^{\ll,\mu}_t,~~~u^{\ll,\mu}_T=0,
\end{equation}
where
\beg{equation}\label{L}
\L f(x)=\int_{\mathbb{R}^{d}\backslash\{0\}}\big[f(x+y)-f(x)-\langle y,\nabla f(x)\rangle\mathds{1}_{\{ |y|\leq 1\}} \big]\nu(\d y), \ \ f\in C_{c}^{\infty}(\mathbb{R}^{d}).
\end{equation}
According to  \cite[Theorem 3.4]{P}, for $\ll>0$  large
enough, \eqref{A10} has a unique solution $u^{\ll,\mu} \in C^1([0,T],C_{b}^{\alpha+\beta}\left(\mathbb{R}^{d};\R^d\right))$ with
\begin{equation}\label{A20}
\|\nn u^{\ll,\mu}\|_{T,\8}\le
\ff{1}{2},
\end{equation}
and
\beg{equation}\label{uuu}
\lambda\|u^{\ll,\mu}\|_{T,\infty}+\sup_{t\in[0,T]}\|\nabla u^{\ll,\mu}_t\|_{\alpha+\beta-1}\leq C\|b\|_{\beta}.
\end{equation}
Let $\theta^{\lambda,\mu}_t(x)=x+u^{\ll,\mu}_t(x)$. By \eqref{E1'}, \eqref{A10},  and using the It\^o formula, we derive
\begin{equation}\label{E-X}
\begin{split}
\d \theta^{\ll,\mu}_t(X_t)&=\ll u^{\ll,\mu}_t(X_t)\d
t+\d L_t \\
& \qquad +\int_{\mathbb{R}^{d}\backslash\{0\}} \left[ u^{\ll,\mu}_t\left(X_{t-}+x \right)-u^{\ll,\mu}_t\left(X_{t-} \right) \right]\tilde{N}(\d t, \d x),\\
\d
\theta^{\ll,\mu}_t(Y_t)
 &=\{\ll
u^{\ll,\mu}_t(Y_t)+ \nn\theta^{\ll,\mu}_t(b_t^{\bar{\mu}}-b_t^{\mu})(Y_t)\}\d
t+\d L_t\\
& \qquad \\
& \qquad +\int_{\mathbb{R}^{d}\backslash\{0\}} \left[ u^{\ll,\mu}_t\left(Y_{t-}+x \right)-u^{\ll,\mu}_t\left(Y_{t-}\right) \right]\tilde{N}(\d t, \d x).
 \end{split}
\end{equation}
Thus, we have
\beg{equation}\label{5-X-X^m-2}\begin{split}
|\theta^{\lambda,\mu}_t(X_t)-\theta^{\lambda,\mu}_t(Y_t)|\leq \sum_{i=1}^3\Lambda_{i}(t),
\end{split}\end{equation}
where
\beg{equation*}\begin{split}
&\Lambda_{1}(t)=\left|\int_{0}^{t}\int_{\mathbb{R}^{d}\backslash\{0\}}\big[u^{\ll,\mu}_s\left(X_{s-}+x \right)-u^{\ll,\mu}_s\left(X_{s-} \right)-u^{\ll,\mu}_s\left(Y_{s-}+x \right)+u^{\ll,\mu}_t\left(Y_{s-} \right)\big]\tilde{N}(\d s, \d x)\right|,\\
&\Lambda_{2}(t)=\int_{0}^{t}\lambda |u^{\ll,\mu}_s(X_s)-u^{\ll,\mu}_s(Y_s)|\d s,\\
&\Lambda_{3}(t)=\int_{0}^{t}|\nn\theta^{\ll,\mu}_s(b_s^{\bar{\mu}}-b_s^{\mu})(Y_s)|\d s.
\end{split}\end{equation*}
Firstly, by {\bf (H2)}, \eqref{A20} and H\"{o}lder inequality, for any $p\geq \kappa$, we obtain
\beg{equation}\begin{split}\label{5-Lambda_4}
\mathbb{E}\sup_{0\leq s\leq t}\Lambda_{3}^{p}(s)
&\leq c_{1}(p,T)\int_0^t\W_\kappa(\bar{\mu}_s,\mu_s)^{p}\d s\leq c_{1}(p,T)\int_0^t\mathbb{E}\sup_{0\leq s\leq r}|X_s-Y_s|^{p}\d r.
\end{split}\end{equation}
Similarly, we have
\beg{equation}\begin{split}\label{5-Lambda_2}
\mathbb{E}\sup_{0\leq s\leq t}\Lambda_{2}^{p}(s)
&\leq c_{2}(p,\lambda,T)\int_0^t\mathbb{E}\sup_{0\leq s\leq r}|X_s-Y_s|^{p}\d r.
\end{split}\end{equation}
Finally, by \cite{HL}, for any $p\geq \kappa$, we have
\beg{equation}\begin{split}\label{5-Lambda_1}
\mathbb{E}\sup_{0\leq s\leq t}\Lambda_{1}^{p}(s)
&\leq  c_{3}(p,T,\nu,\alpha,\beta)\int_0^t\mathbb{E}\sup_{s\in[0,r]} \left| X_{s}-Y_{s} \right|^{p}\d r.
\end{split}\end{equation}
Combining formulas \eqref{5-X-X^m-2}--\eqref{5-Lambda_1} and \eqref{A20}, we get
\beg{equation*}\begin{split}
\mathbb{E}\sup_{0\leq s\leq t}|X_{s}-Y_{s} |^{p}
&\leq C\int_0^t\mathbb{E}\sup_{s\in[0,r]} \left| X_{s}-Y_{s} \right|^{p}\d r.
\end{split}\end{equation*}
Since $X-Y$ is a bounded process, Gronwall's inequality implies that $\P$-a.s. $X_t=Y_t$ for all $t\in [0,T].$
\end{proof}
 \begin{rem}\label{PDEs} Although the PDE considered in \cite{P} is elliptic, the PDE \eqref{A10} is parabolic, we can obtain \eqref{A20} and \eqref{uuu} by the same method in \cite{P} since both the elliptic and parabolic PDEs have similar probability representation.
\end{rem}
\section{The convergence rate of EM Scheme for SDEs}
In this section, we exploit the
stochastic interacting particle systems to approximate \eqref{E1}. Let
$N\ge1$ be an integer and $(X_0^i,L^i_t)_{1\le i\le N}$ be i.i.d.
copies of $(X_0,L_t).$ Consider the following stochastic
non-interacting particle systems
\begin{equation}\label{C4}
\d X_t^i=b(X_t^i,\mu_t^i)\d t+\d
L_t^i,~~~t\ge0,~~~i\in\mathcal {S}_N:=\{1,\cdots,N\}
\end{equation}
with $\mu_t^i:=\mathscr{L}_{X_t^i}$. By the weak uniqueness, we have $\mu_t=\mu_t^i, i\in\mathcal
{S}_N.$ Let $\delta_x$ be  Dirac's delta measure centered
at the point $x\in\R^d$ and $ \tt\mu_t^N $ be the empirical distribution associated
with $X_t^1,\cdots,X_t^N$, i.e.,
\begin{equation}\label{H1}
\tt\mu_t^N =\ff{1}{N}\sum_{j=1}^N\dd_{X_t^j}.
\end{equation}
Moreover, the stochastic $N$-interacting particle
systems is defined:
\begin{equation}\label{eq4}
\d X_t^{i,N}=b(X_t^{i,N},\hat\mu_t^N)\d t+\d
L_t^i,~t\ge0,~X_0^{i,N}=X_0^i,~i\in\mathcal {S}_N,
\end{equation}
where  $\hat\mu_t^N$ means the empirical distribution corresponding
to  $X_t^{1,N},\cdots,X_t^{N,N}$, namely,
\begin{equation*}
 \hat\mu_t^N :=\ff{1}{N}\sum_{j=1}^N\dd_{X_t^{j,N}}.
 \end{equation*}
We remark that particles $(X^i)_{i\in\mathcal {S}_N}$ are mutually
independent and that particles $(X^{i,N})_{i\in\mathcal {S}_N}$ are
interacting and are not independent.

Let $\lfloor
a\rfloor$ stipulates the integer part of $a\ge0.$ To discretize \eqref{eq4} in time, we introduce the continuous time
EM scheme defined as below: for any  $\dd\in(0,\e^{-1}),$
\begin{equation}\label{C3}
\d X_t^{\dd,i,N}=b(X_{t_\dd}^{\dd,i,N},\hat\mu_{t_\dd}^{\dd,N})\d
t+\d
L_t^i,~~~t\ge0,~~~X_0^{\dd,i,N}=X_0^{i,N},
\end{equation}
where $t_\dd:=\lfloor t/\dd\rfloor\dd$  and
 $$\hat\mu_{k\dd}^{\dd,N} :=\ff{1}{N}\sum_{j=1}^N\dd_{X_{k\dd}^{\dd,j,N}},~~~~k\ge0.$$

The following result states  that the continuous time EM scheme
corresponding to stochastic interacting particle systems converges
strongly to the non-interacting particle system whenever  the
particle number goes to infinity and the stepsize approaches to zero
and moreover provides the convergence rate.

\begin{thm}\label{th3}
 Assume {\bf(H1)}-{\bf(H2)} and suppose further $\mathscr{L}_{X_0}\in  \scr
P_p$ for some $p\in[\kappa,\alpha)$.  Then, for any $T>0$ and $q\in(p,\alpha)$, there exists a
constant $C>0$ depending on $p,d,q,T$ and $\sup_{t\in[0,T]}\E|X_t^i|^q$ such that
\begin{equation*}
\sup_{i\in\mathcal {S}_N}\E\Big(\sup_{0\le t\le
T}|X_t^i-X_t^{\dd,i,N}|^p\Big)\le C
\begin{cases}
\dd^{\frac{p\beta}{\aa}}+N^{-\frac{1}{2}}+N^{\frac{p}{q}-1}, ~~~~~~~~~~~~~~~p>\frac{d}{2},\ \ q\neq 2p, \\
\dd^{\frac{p\beta}{\aa}}+N^{-\frac{1}{2}}\log (1+N)+N^{\frac{p}{q}-1}, ~p=\frac{d}{2},\ \ q\neq 2p\\
\dd^{\frac{p\beta}{\aa}}+N^{-\frac{2}{d}}+N^{\frac{p}{q}-1},
  ~~~~~~~~~~~~~~~p\in(0,\frac{d}{2}),\ \ q\neq \frac{d}{d-p}.
\end{cases}
\end{equation*}
\end{thm}
Firstly, under {\bf(H1)}-{\bf(H2)}, the stochastic $N$-interacting particle systems
\eqref{eq4} are strongly wellposed, see Lemma \ref{lem} below.
\begin{lem}\label{lem}
 Assume that {\bf(H1)}   and {\bf(H2)}
hold. Then for any $\F_0$-measurable random variable $X_0$ with $\mathscr{L}_{X_0}\in  \scr
P_p$ for some $p\in(0,\alpha)$, \eqref{eq4} admits a strong solution satisfying 
$$\sup_{i\in\mathcal {S}_N}\E|X_t^{i,N}|^p<\infty,\ \ t>0.$$
\end{lem}

\begin{proof}
For $x:=(x_1,\cdots,x_N)^*\in(\R^d)^N$,   $x_i\in\R^d$, set
\begin{equation*}
\tt\mu^N_x:=\ff{1}{N}\sum_{i=1}^N\dd_{x_i}, ~~~\hat b(x):=(b(x_1,\tt\mu^N_x),\cdots,b(x_N,\tt\mu^N_x))^*, \ \
\hat L_t:=(L_t^1,\cdots,L_t^N)^*.
\end{equation*}
Obviously, $(\hat L_t)_{t\ge0}$ is an $Nd$-dimensional L\'{e}vy process. Then, \eqref{eq4} can be reformulated as
\begin{equation}\label{B1}
\d X_t=\hat b(X_t)\d t+\d \hat{L}_t,~~~t\ge0.
\end{equation}
Firstly, $\hat{L}_t$ has symbol $\tilde{\Psi}(u)=\sum_{i=1}^N\Psi(u_i)$ for $u=(u_1,\cdots,u_N)^\ast\in(\R^d)^N$. Clearly, ({\bf H1}) holds for $\tilde{\Psi}(u)$ for some constant $C(N,\alpha)>0$ since $\Psi$ satisfies ({\bf H1}) and the inequality
$$C_\alpha\sum_{i=1}^N|u_i|^\alpha\geq C(N,\alpha)\left(\sum_{i=1}^N|u_i|^2\right)^{\frac{\alpha}{2}}$$
holds for some constant $C(N,\alpha)>0$.
By ({\bf H2}), a straightforward
calculation shows that
\begin{equation}\label{B2}
\begin{split}
|\hat b(x)|\leq
C_{N},~~~x\in(\R^d)^N
\end{split}
\end{equation}
for some constant $C_{N}>0$. Observe that
\begin{equation}\label{W4}\ff{1}{N}\sum_{j=1}^N(\dd_{x_j}\times\dd_{y_j})\in\mathcal
{C}(\tt\mu^N_x,\tt\mu^N_y),~~~~x_j,y_j\in\R,\end{equation} so that we
have
\begin{equation}\label{W3}
\mathbb{W}_\kappa(\tt\mu^N_x,\tt\mu^N_y)\le
\left(\ff{1}{N}\sum_{j=1}^N|x_j-y_j|^\kappa\right)^{\frac{1}{\kappa}}.
\end{equation}
This together with ({\bf H2}) and H\"{o}lder inequality
implies that
\begin{equation}\label{B3}
\begin{split}
|\hat b(x)-\hat b(x')|\le \hat C_{N}\{|x-x'|+|x-x'|^\bb\},
\end{split}
\end{equation}
for some constant $\hat C_{N}>0$. Thus, according to \cite{P},
\eqref{eq4} has a unique strong solution. Finally, the estimate follows from the fact $\E|L_t|^p<\infty$ for any $p\in(0,\alpha)$ and the boundedness of $b$.
\end{proof}

\subsection{Proof of Theorem \ref{th3}}
The proof of Theorem \ref{th3} is based on two  lemmas below, where
the first one is concerned with  propagation of chaos for
McKean-Vlasov SDEs with irregular drift coefficients.
We state it as follows.
\begin{lem}\label{pro3}
Under the assumptions of Theorem \ref{th3}, then for any $T>0$ and $q\in(p,\alpha)$, there
exists a constant $C>0$ depending on $p,d,q,T$ and $\sup_{t\in[0,T]}\E|X_t^i|^q$ such that
\begin{equation*}
\sup_{i\in\mathcal {S}_N}\E\Big(\sup_{0\le t\le
T}|X_t^i-X_t^{i,N}|^p\Big)\le C
\begin{cases}
N^{-\frac{1}{2}}+N^{\frac{p}{q}-1}, ~~~~~~~~~~~~~~~~~p>\frac{d}{2},\ \ q\neq 2p, \\
N^{-\frac{1}{2}}\log (1+N)+N^{\frac{p}{q}-1}, ~~~  p=\frac{d}{2},\ \ q\neq 2p
 \\
 N^{-\frac{2}{d}}+N^{\frac{p}{q}-1},
  ~~~~~~~~~~~~~~~~~p\in(0,\frac{d}{2}),\ \ q\neq \frac{d}{d-p}
 \end{cases}
\end{equation*}
\end{lem}

\begin{proof}
For any $i\in\mathcal {S}_N$ and $x\in\R^d$, let
$b_t^{\mu^i}(x)=b(x,\mu_t^i)$ and
$b_t^{\hat\mu^N}=b(x,\hat\mu_t^N)$. Then, \eqref{C4} and \eqref{eq4}
can be rewritten respectively  as
\begin{equation*}
\begin{split}
\d X_t^i&=b_t^{\mu^i}(X_t^i)\d t+\d L_t^i,\\
\d X_t^{i,N}&=b_t^{\hat\mu^N}(X_t^{i,N})\d t+\d L_t^i.
\end{split}
\end{equation*}
For  $\ll>0 $, consider the following PDE for
$u^{\ll,\mu^i}:[0,T]\times\R^d\to\R^d$:
\begin{equation}\label{A1}
\partial_tu^{\ll,\mu^i}_t+\L u^{\ll,\mu^i}_t+\nn_{b_t^{\mu^i}}u^{\ll,\mu^i}_t+
b_t^{\mu^i}=\ll u^{\ll,\mu^i}_t,~~~u^{\ll,\mu^i}_T=0,
\end{equation}
where $\L$ is defined in \eqref{L}. Since $\mu^i=\mu$ for any $i\in\mathcal
{S}_N$, there exists large
enough $\ll>0$ independent of $i$, such that \eqref{A1} has a unique solution $u^{\ll,\mu^i} \in C^1([0,T],C_{b}^{\alpha+\beta}\left(\mathbb{R}^{d},\R^d\right))$, which is equal to $u^{\ll,\mu}$. Moreover, \eqref{A20} and \eqref{uuu} hold.

Applying It\^o's
formula to $\theta^{\ll,\mu^i}_t(x):=x+u^{\ll,\mu^i}_t(x),
x\in\R^d$ yields
\begin{equation}\label{S2}
\begin{split}
\d \theta^{\ll,\mu^i}_t(X_t^i)&=\ll u^{\ll,\mu^i}_t(X_t^i)\d
t+\d L^i_t \\
& \qquad +\int_{\mathbb{R}^{d}\backslash\{0\}} \left[ u^{\ll,\mu^i}_t\left(X_{t-}^i+x \right)-u^{\ll,\mu^i}_t\left(X_{t-}^i \right) \right]\tilde{N}(\d t, \d x),\\
\d
\theta^{\ll,\mu^i}_t(X_t^{i,N})
 &=\{\ll
u^{\ll,\mu^i}_t(X_t^{i,N})+\d L^i_t+ \nn\theta^{\ll,\mu^i}_t(b_t^{\hat\mu^N}-b_t^{\mu^i})(X_t^{i,N})\}\d
t\\
& \qquad \\
& \qquad +\int_{\mathbb{R}^{d}\backslash\{0\}} \left[ u^{\ll,\mu^i}_t \left(X_{t-}^{i,N}+x \right)-u^{\ll,\mu^i}_t\left(X_{t-}^{i,N}\right) \right]\tilde{N}(\d t, \d x).
 \end{split}
\end{equation}
For simplicity, set
$\Lambda^{\ll,i,N}_t=\theta^{\ll,\mu^i}_t(X_t^i)-\theta^{\ll,\mu^i}_t(X_t^{i,N})$. We have
\begin{equation*}
\begin{split}
|\Lambda^{\ll,i,N}_t|&\le
\lambda\int_0^t |u^{\ll,\mu^i}_s(X_s^i)-u^{\ll,\mu^i}_s(X_s^{i,N})|\d s+\int_0^t|(\nn\theta^{\ll,\mu^i}_s(b_s^{\hat\mu^N}-b_s^{\mu^i}))(X_s^{i,N})|\d s\\
&\quad+\Bigg|\int_0^t\int_{\mathbb{R}^{d}\backslash\{0\}}\bigg( \left[ u^{\ll,\mu^i}_s\left(X_{s-}^i+x \right)-u^{\ll,\mu^i}_s\left(X_{s-}^i \right) \right]\\
&\qquad\qquad \qquad-\left[ u^{\ll,\mu^i}_s \left(X_{s-}^{i,N}+x \right)-u^{\ll,\mu^i}\left(X_{s-}^{i,N}\right) \right]\bigg)\tilde{N}(\d s, \d x)\Bigg|\\
&=:I_{1,i}(t)+I_{2,i}(t)+I_{3,i}(t).
\end{split}
\end{equation*}
Completely the same with \eqref{5-Lambda_1}, we have
$$\E\sup_{s\in[0,t]}|I_{3,i}(s)|^p\leq \int_0^t\mathbb{E}\sup_{s\in[0,r]}| X_s^{i}-X_s^{i,N}|^{p}\d r.$$
Next, by assumption {\bf (H2)} and H\"{o}lder inequality, for any $p\geq \kappa$, we obtain
\begin{equation}\label{A4}
\begin{split}
\mathbb{E}\sup_{0\leq s\leq t}|I_{2,i}(s)|^{p}&\le
C_2\E\int_0^t\mathbb{W}_\kappa(\hat\mu^N_s,\tt\mu^N_s)^p+\mathbb{W}_\kappa(\tt\mu^N_s,\mu^i_s)^p\d
s\\
&\le
C_2\int_0^t\{\E\sup_{s\in[0,r]}| X_s^{i}-X_s^{i,N}|^{p}+\E\mathbb{W}_\kappa(\tt\mu^N_r,\mu^i_r)^p\}\d
r\\
&\le
C_2\int_0^t\{\E\sup_{s\in[0,r]}| X_s^{i}-X_s^{i,N}|^{p}+\E\mathbb{W}_p(\tt\mu^N_r,\mu^i_r)^p\}\d
r
\end{split}
\end{equation}
Similarly, we have
\beg{equation}\begin{split}\label{5-Lambda_20}
\mathbb{E}\sup_{0\leq s\leq t}|I_{1,i}(s)|^p
&\leq c_{2}(p,\lambda,T)\int_0^t\mathbb{E}\sup_{0\leq s\leq r}| X_s^{i}-X_s^{i,N}|^{p}\d r.
\end{split}\end{equation}
Thus, we
find that for some constant $C_{2,\ll}>0,$
\begin{equation*}
\E\Big(\sup_{0\le s\le t}|\Lambda^{\ll,i,N}_s|^p\Big)\le
C_{2,\ll}\int_0^t\{\mathbb{E}\sup_{0\leq s\leq r}| X_s^{i}-X_s^{i,N}|^{p}+\E\mathbb{W}_p(\tt\mu^N_r,\mu^i_r)^p\}\d
r.
\end{equation*}
Set $Z_t^{i,N}=X_t^i-X_t^{i,N}$
for convenience. This, together with the facts that $ |Z_t^{i,N}|^p\le
2^p|\Lambda^{\ll,i,N}_t|^p $ due to \eqref{A20},  leads to
\begin{equation*}
\E\Big(\sup_{0\le s\le t}|Z_s^{i,N}|^p\Big)\le
C_{3,\ll}\int_0^t\{\mathbb{E}\sup_{0\leq s\leq r}| X_s^{i}-X_s^{i,N}|^{p}+\E\mathbb{W}_p(\tt\mu^N_r,\mu^i_r)^p\}\d
r
\end{equation*}
for some constant $C_{3,\ll}>0$. On the other hand, according to \cite[Theorem 1]{FG}, for any $q\in(p,\alpha)$,
\begin{equation}\label{G1}\sup_{0\le t\le
T}\E\mathbb{W}_p(\tt\mu^N_t,\mu^i_t)^p\leq C_4
\begin{cases}
N^{-\frac{1}{2}}+N^{\frac{p}{q}-1}, ~~~~~~~~~~~~~~~~~~p>\frac{d}{2},\ \ q\neq 2p, \\
N^{-\frac{1}{2}}\log (1+N)+N^{\frac{p}{q}-1}, ~~~  p=\frac{d}{2},\ \ q\neq 2p
 \\
 N^{-\frac{2}{d}}+N^{\frac{p}{q}-1},
  ~~~~~~~~~~~~~~~~~~p\in(0,\frac{d}{2}),\ \ q\neq \frac{d}{d-p}
 \end{cases}
\end{equation}
holds for some constant $C_4>0$ depending on $p,d,q$ and $\sup_{t\in[0,T]}\mu^i_t(|\cdot|^q)$. 
Hence, due to the boundedness of $X^{i}-X^{i,N}$, the desired assertion follows from Gronwall's inequality. 
\end{proof}
\begin{rem}\label{moment} Noting that for $q\geq \alpha$, $\mu_t^i(|\cdot|^q)=\infty$, the condition in \cite[Theorem
5.8]{CD} does not hold. So we adopt \cite[Theorem 1]{FG} in place of \cite[Theorem
5.8]{CD} used in \cite{BH}.
\end{rem}
The next lemma gives the estimate for $|X_t^{\dd,i,N}-X_{t_\delta}^{\dd,i,N}|$, which is useful in the sequel.
\beg{lem}\label{L-5-t-td} Assume {\bf (H1)} and {\bf (H2)}, then for any $0<p<\alpha$, $t\in[0,T]$,
\begin{equation*}
\sup_{i\in\mathcal {S}_N}\mathbb{E}\left|X_t^{\dd,i,N}-X_{t_\delta}^{\dd,i,N}\right|^p\leq C(p,\nu)\delta^{\frac{p}{\alpha}}
\end{equation*}
holds for some constant $C(p,\nu)$ depending on $p$ and $\nu$.
\end{lem}
\begin{proof}[Proof]
The result follows immediately from \eqref{C3}, the boundedness of $b$, the scaling property of $L$ and $\E |L_t|^p<\infty$ for $p\in(0,\alpha)$.
\end{proof}

\begin{lem}\label{D1}
Under the assumptions of Theorem \ref{th3},  then for any $T>0$ and $q\in(p,\alpha)$, there
exists a constant $C>0$ depending on $p,d,q,T$ and $\sup_{t\in[0,T]}\E|X_t^i|^q$ such that
\begin{equation*}
\sup_{i\in\mathcal {S}_N}\E\Big(\sup_{0\le t\le
T}|X_t^{i,N}-X_t^{\dd,i,N}|^p\Big)\le C
\begin{cases}
\dd^{\frac{p\beta}{\aa}}+N^{-\frac{1}{2}}+N^{\frac{p}{q}-1}, ~~~~~~~~~~~~~~~p>\frac{d}{2},~ q\neq 2p, \\
\dd^{\frac{p\beta}{\aa}}+N^{-\frac{1}{2}}\log (1+N)+N^{\frac{p}{q}-1}, ~p=\frac{d}{2},~q\neq 2p,\\
\dd^{\frac{p\beta}{\aa}}+N^{-\frac{2}{d}}+N^{\frac{p}{q}-1},
  ~~~~~~~~~~~~~~~p\in(0,\frac{d}{2}),~q\neq \frac{d}{d-p}.
\end{cases}
\end{equation*}

\end{lem}

\begin{proof} For $x\in\R^d$ and $i\in\mathcal {S}_N$,
let $b^{\hat\mu^{\dd,N}}_{k\dd}(x)=b(x,\hat\mu^{\dd,N}_{k\dd})$ so
that \eqref{C3} can be reformulated as
\begin{equation*}
\d X_t^{\dd,i,N}=b^{\hat\mu^{\dd,N}}_{t_\dd}(X_{t_\dd}^{\dd,i,N})\d
t+\d L_t^i.
\end{equation*}
Let $u^{\ll,\mu^i}$ be the solution to \eqref{A1}. Again applying It\^o's formula to
$\theta^{\ll,\mu^i}_t(x)=x+u^{\ll,\mu^i}_t(x)$ gives that
\begin{equation}\label{S3}
\begin{split}
\d \theta^{\ll,\mu^i}_t( X_t^{\dd,i,N})&=\Big\{\ll u^{\ll,\mu^i}_t(
X_t^{\dd,i,N})  +\nn\theta^{\ll,\mu^i}_t(
X_t^{\dd,i,N})(b^{\hat\mu^{\dd,N}}_{t_\dd}(X_{t_\dd}^{\dd,i,N})-b^{ \mu^i}_t(X_t^{\dd,i,N}))+\d L_t^i\\
&\quad
\int_{\mathbb{R}^{d}\backslash\{0\}} \left[ u^{\ll,\mu^i}_t \left(X_{t-}^{\delta,i,N}+x \right)-u^{\ll,\mu^i}_t\left(X_{t-}^{\delta,i,N}\right) \right]\tilde{N}(\d t, \d x).
\end{split}
\end{equation}
Set
\begin{equation*}
\Theta^{\ll,i,N}_t:=\theta^{\ll,\mu^i}_t(X_t^{i,N})-\theta^{\ll,\mu^i}_t(X_t^{\dd,i,N}),~~~~
 Z_t^{\dd,i,N}:=X_t^{i,N}-X_t^{\dd,i,N}.
\end{equation*}
Then, for any $p\in[\kappa,\alpha)$, from \eqref{S3} and    the second SDE in \eqref{S2}, we
deduce from H\"older's inequality that
\begin{align*}
&\E\Big(\sup_{0\le s\le t} |\Theta^{\ll,i,N}_s|^p\Big)\\&\le
C_{\ll,p,T}\Bigg\{\int_0^t\E\left| u^{\ll,\mu^i}_s(
X_s^{\dd,i,N}) -u^{\ll,\mu^i}_s(
X_s^{i,N})\right|^p\d s\\
&\quad+\int_0^t\E\left|\nn\theta^{\ll,\mu^i}_s(b_s^{\hat\mu^N}-b_s^{\mu^i})(X_s^{i,N})-\nn\theta^{\ll,\mu^i}_s(
X_s^{\dd,i,N})(b^{\hat\mu^{\dd,N}}_{s_\dd}(X_{s_\dd}^{\dd,i,N})-b^{
\mu^i}_s(X_s^{\dd,i,N}))\right|^p\d s\\
&\quad+\E\sup_{r\in[0,t]}\Bigg|\int_0^r\int_{\mathbb{R}^{d}\backslash\{0\}} \Bigg(\left[ u^{\ll,\mu^i}_s \left(X_{s-}^{\delta,i,N}+x \right)-u^{\ll,\mu^i}_s\left(X_{s-}^{\delta,i,N}\right) \right]\\
&\qquad\qquad\qquad\qquad\qquad\qquad-\left[ u^{\ll,\mu^i}_s \left(X_{s-}^{i,N}+x \right)-u^{\ll,\mu^i}\left(X_{s-}^{i,N}\right) \right]\Bigg)\tilde{N}(\d s, \d x)\Bigg|^p\Bigg\}\\
&=:C_{\ll,p,T}\{J_1(t)+J_2(t)+J_3(t)\}
\end{align*}
for some constant $C_{\ll,p,T}>0.$ In what follows, we intend to estimate
$J_i(t), i=1,2,3$, one-by-one. Owing to \eqref{A20} and \eqref{5-Lambda_1}, there exists a
constant $c_1>0$ such that
\begin{equation}\label{L1}
J_1(t)+J_3(t)\le c_1\int_0^t\E \sup_{s\in[0,r]}|Z_s^{\dd,i,N}|^p\d r.
\end{equation}
It remains to estimate $J_2(t)$. By \eqref{A20}, we arrive at
\begin{equation}\label{A100}
\begin{split}
J_2(t)&\le
c_2\int_0^t\{\E\mathbb{W}_\kappa(\mu_s^i,\hat\mu^N_s)^p+\E|X_s^{\dd,i,N}-X_{s_\dd}^{\dd,i,N}|^{p\beta} +\E\mathbb{W}_\kappa(\mu_s^i,\hat\mu^{\dd,N}_{s_\dd})^p\}\d
s\\
&\le
c_3\int_0^t\{\dd^{\frac{p\beta}{\aa}}+\E\mathbb{W}_\kappa(\mu_s^i,\tt\mu^N_s)^p+\E\mathbb{W}_\kappa(\tt\mu^N_s,\hat\mu^N_s)^p
  +\E\mathbb{W}_\kappa(\tt\mu^N_s,\hat\mu^{\dd,N}_{s_\dd})^p\}\d
s\\
\end{split}
\end{equation}
for some constants $c_2,c_3>0$, where we have used Lemma \ref{L-5-t-td}.
On the other hand,  similarly to  \eqref{W3}, we obtain from Lemma \ref{L-5-t-td}
\begin{equation}\label{H2}
\begin{split}
&\E\mathbb{W}_\kappa(\tt\mu^N_t,\hat\mu^N_t)^p
  +\E\mathbb{W}_\kappa(\tt\mu^N_t,\hat\mu^{\dd,N}_{t_\dd})^p\\
  &\le\ff{1}{N}
  \sum_{j=1}^N\{\E|X_t^j-X_t^{j,N}|^p+\E|X_t^j-X_{t_\dd}^{\dd,j,N}|^p\}\\
  &\le C_{1,T}\dd^{\frac{p}{\alpha}}+\E|X_t^i-X_t^{i,N}|^p+c(p)\E|X_t^{i,N}-X_t^{\dd,i,N}|^p
\end{split}
\end{equation}
for some $C_{1,T}>0,$
where in the last display we used the facts that
$(X^j-X^{j,N})_{j\in\mathcal {S}_N}$ and
$(X^j-X^{\dd,j,N})_{j\in\mathcal {S}_N}$ are identically
distributed. Then, plugging \eqref{H2} back into \eqref{A100} gives
that
\begin{equation}\label{L2}
J_2(t) \le
C_{2,T}\int_0^t\{\dd^{\frac{p\beta}{\alpha}}+\E\mathbb{W}_p(\mu_s^i,\tt\mu^N_s)^p+\E|X_s^i-X_s^{i,N}|^p+
\E\sup_{r\in[0,s]}|Z_r^{\dd,i,N}|^p\}\d s
\end{equation}
for some constant $C_{2,T}>0$.
Now, combining \eqref{L1}, \eqref{L2}, we arrive at
\begin{equation*}
\E\Big(\sup_{0\le s\le t} |\Theta^{\ll,i,N}_s|^p\Big)\le C_{4,T}
\int_0^t\{\dd^{\frac{p\beta}{\aa}}+\E\mathbb{W}_p(\mu_s^i,\tt\mu^N_s)^p+\E|X_s^i-X_s^{i,N}|^p+
\E\sup_{r\in[0,s]}|Z_r^{\dd,i,N}|^p\}\d s
\end{equation*}
for some constant $C_{4,T}>0.$ This, together with $
|Z_t^{\dd,i,N}|^p\le 2^p|\Theta^{\ll,i,N}_t|^p $ due to \eqref{A20},
yields
\begin{equation*}
\E\Big(\sup_{0\le s\le t} |Z_s^{\dd,i,N}|^p\Big)\le C_{5,T}
\int_0^t\{\dd^{\frac{p\beta}{\aa}}+\E\mathbb{W}_p(\mu_s^i,\tt\mu^N_s)^p+\E|X_s^i-X_s^{i,N}|^p+
\E\sup_{r\in[0,s]}|Z_r^{\dd,i,N}|^p\}\d s
\end{equation*}
for some constant $C_{5,T>0}$. Consequently, the desired assertion holds true by  applying Gronwall's inequality and
employing Lemma \ref{pro3} and \eqref{G1}.

\end{proof}
\begin{proof}[Proof of Theorem \ref{th3}] Theorem \ref{th3} immediately follows from Lemma \ref{pro3} and Lemma \ref{D1}. 
\end{proof}


\beg{thebibliography}{99}
\bibitem{ABRS}  H. Airachid,     M. Bossy,    C. Ricci,  l. Szpruch,
\emph{New particle representations for ergodic McKean-Vlasov SDEs,}
arXiv:1901.05507.

\bibitem{BHY}J.  Bao, X. Huang,  C. Yuan, \emph{Convergence Rate of Euler--Maruyama Scheme for SDEs
with H\"older--Dini Continuous Drifts, } { J.  Theor. Probab.},
{ 32} (2019),  848--871.

\bibitem{BH} J. Bao, X. Huang \emph{Approximations of Mckean-Vlasov SDEs with Irregular Coefficients,} arXiv:1905.08522.

\bibitem{BMP}
 M. Bauer,   T. Meyer-Brandis,   F.  Proske,  \emph{Strong solutions of mean-field stochastic differential equations with irregular
 drift,} { Electron. J. Probab.},
  23 (2018), paper no. 132, 35 pp.

\bibitem{BLM} R. Buckdahn, J. Li,   J. Ma, \emph{A mean-field stochastic
control problem with partial observations, } {Ann. Appl.
Probab.}, {27} (2017),  3201--3245.

 \bibitem{BLPR} R. Buckdahn,  J. Li,   S. Peng,  C. Rainer,
 \emph{ Mean-field stochastic differential equations and associated PDEs, } {Ann. Probab.}, {45} (2017),  824--878.

\bibitem{CD} R. Carmona,  F. Delarue,   \emph{Probabilistic theory of mean field games with applications I,}
vol. 84 of Probability Theory and Stochastic Modelling, Springer International Publishing,
1st ed., 2017.


\bibitem{Ch}P.-E. Chaudru de Raynal,  \emph{Strong well-posedness of McKean-Vlasov stochastic differential equation with H\"older
drift, } arXiv:1512.08096v2.


\bibitem{CM} D. Crisan,  E. McMurray,  \emph{Smoothing properties of
McKean-Vlasov SDEs,}  {Probab. Theory Related Fields}, { 171}
(2018),   97--148.

\bibitem{DEG} G. dos Reis,   S. Engelhardt, G. Smith,   \emph{Simulation of McKean-Vlasov SDEs with super linear
growth,}   arXiv:1808.05530.

\bibitem{EGZ}  A. Eberle,  A. Guillin, R. Zimmer,
\emph{Quantitative Harris--type theorems for diffusions and McKean-Vlasov
processes,} { Trans. Amer. Math. Soc.}, { 371} (2019),
7135--7173.

\bibitem{FG} N. Fournier, A. Guillin, \emph{On the rate of convergence in Wasserstein distance of the empirical measure,} arXiv:1312.2128. 

 \bibitem{GR} I. Gy\"ongy, M.  R\'{a}sonyi,  \emph{A note on Euler approximations for SDEs with H\"older continuous diffusion
 coefficients,}
  { Stochastic Process. Appl.}, {121} (2011),   2189--2200.

\bibitem{HL} X. Huang, Z. Liao, \emph{The Euler-Maruyama method for S(F)DEs with H\"{o}lder drift and $\alpha$-stable noise,} Stoch. Anal. Appl., 36 (2018), 28-39.

\bibitem{HW} X. Huang, F.-Y.  Wang,  \emph{Distribution dependent SDEs
with singular coefficients, Stochastic Process, }
Appl., 129 (2019), 4747-4770.

\bibitem{H1} X. Huang, \emph{Path-Distribution Dependent SDEs with Singular Coefficients,} arXiv:1805.01682.

\bibitem{IW} N.  Ikeda, S. Watanabe,  \emph{Stochastic differential equations and diffusion processes}, North-Holland Mathematical Library, 24.
 North-Holland Publishing Co., Amsterdam-New York; Kodansha, Ltd., Tokyo, 1981.

\bibitem{J} P. Jin, \emph{On weak solutions of SDEs with singular time-dependent drift and driven by stable processes,}  Stoch. Dyn. 18 (2018).

\bibitem{K} H. Kunita, \emph{Stochastic differential equations based on L\'{e}vy processes and stochastic flows of diffeomorphisms,} Birkh\"{a}user Boston, 70 (7):305-373, (2004).

\bibitem{L} E. Lenglart, \emph{Relation de domination entre deux processus,} Annales de l'Institut Henri Poincar\'{e}. Section B. Calcul des Probabilit\'{e}s et Statistique. Nouvelle S\'{e}rie, (2):171-179, (1977).

\bibitem{LM}  J. Li, H.  Min, \emph{Weak solutions of mean-field stochastic
differential equations and application to zero-sum stochastic
differential games,}
 { SIAM J. Control Optim.}, {54} (2016),   1826--1858.


\bibitem{MX} R. Mikulevi\u{C}ius, F. H. Xu, \emph{On the rate of convergence of strong Euler approximation for SDEs driven by Levy processes,}  arXiv:1608.02303.


\bibitem{Mc} H. P. McKean, Jr.,  \emph{A class of Markov processes associated with nonlinear parabolic
equations,} { Proc. Nat. Acad. Sci. U.S.A.}, {56} (1966),
1907--1911.

\bibitem{MS}  S. Mehri,       W. Stannat,  \emph{Weak solutions to Vlasov--McKean equations under
Lyapunov--typeconditions,  }   arXiv:1901.07778.

\bibitem{MV} Yu. S. Mishura, A. Yu. Veretennikov, \emph{Existence and uniqueness theorems for solutions of McKean-Vlasov stochastic equations,} arXiv:1603.02212.

\bibitem{NT2}  H.-L. Ngo, D. Taguchi, \emph{ Strong rate of convergence for the
Euler-Maruyama approximation of stochastic differential equations
with irregular coefficients,} { Math. Comput.},{ 85} (2016),
1793--1819.
\bibitem{PT} O. M. Pamen, D. Taguchi, \emph{Strong rate of convergence for the Euler-Maruyama approximation of SDEs with H\"{o}lder continuous drift coefficient,}  arXiv:1508.07513.

\bibitem{P} E. Priola, \emph{Pathwise uniqueness for singular SDEs driven by stable processes,} Osaka Journal of Mathematics, 49 (2010), 421-447.

\bibitem{RW} P. Ren, F.-Y. Wang, \emph{Bismut Formula for Lions Derivative of Distribution Dependent SDEs and Applications,} arXiv:1809.06068.

\bibitem{RZ} M. R\"{o}ckner, X. Zhang, \emph{Well-posedness of distribution dependent SDEs with singular drifts,} arXiv:1809.02216.

\bibitem{SZ} A.-S.  Sznitman,  \emph{Topics in propagation of chaos},
Springer, 1991.

\bibitem{So} Y. Song, \emph{Gradient Estimates and Exponential Ergodicity for Mean-Field SDEs with Jumps,} to appear in JTP.

\bibitem{Wangb} F.-Y. Wang,  \emph{Distribution dependent SDEs for Landau type equations,} { Stochastic Process. Appl.}, {128} (2018),   595--621.

\bibitem{W19} F.-Y. Wang,  \emph{Ergodicity and Feyman-Kac Formula for Space-Distribution Valued Diffusion
Processes,} arXiv:1904.06795.

\bibitem{Y} L. Yan, \emph{The Euler scheme with irregular coefficients}, Annals of Probability, 30 (2002), 1172-1194.

\bibitem{Z2} X. Zhang, \emph{A discretized version of Krylov's estimate and its applications,} arXiv:1909.09976.

\bibitem{Z1} X. Zhang, \emph{Stochastic differential equations with Sobolev drifts and driven by $\alpha$-stable processes,} Ann I H Poinc\'{a}re-PR, 49(2013), 1057-1079.

\bibitem{ZV} A. K. Zvonkin,  \emph{A transformation of the phase space of a diffusion process that removes the drift,}  Math. Sb. 93 (135)(1974).
\end{thebibliography}
\end{document}